\documentclass[12pt]{amsart}
\usepackage[utf8]{inputenc}

\usepackage{amsmath,amsfonts,euscript,amssymb,amsthm,a4,times,color}

\usepackage[dvipsnames]{xcolor}

\usepackage{amssymb,amsmath,amsfonts,amsthm,latexsym,enumerate, esint, manyfoot}
\usepackage{graphics,graphicx}
\usepackage[english]{babel}

\usepackage[colorlinks,linkcolor = blue,citecolor = red,pagebackref=false]{hyperref}
\usepackage{ulem}
\usepackage{amsfonts}
\usepackage{amsxtra}
\usepackage[abbrev]{amsrefs}
\usepackage{mathtools}
\usepackage{tikz-cd}
\usepackage{float}

\newcommand{\dive}{\mathrm{div}\,}

\newcommand{\loc}{\mathrm{loc}}

\def\real{\mathbb R}

\def\R{\mathbb R}

\def\proofof#1{\begin{proof}[Proof of #1]}

\def\d{\mathrm{d}}

\def\E{\mathcal{E}}
\def\dx{\d x}

\def\={^{\wedge}}

\def\eqn#1$$#2$${\begin{equation}\label#1#2\end{equation}}

\def\mir1{\mathcal L_1}

\def\loc{\operatorname{loc}}

\def\Sym{\real^{n\times n}_{\textup{sym}}}

\newtheorem{thm}{Theorem}[section]

\newtheorem{lem}[thm]{Lemma}

\newtheorem{prop}[thm]{Proposition}

\theoremstyle{definition}
\newtheorem{defn}[thm]{Definition}

\newcommand{\xint}[3]{{\setbox0=\hbox{$#1{#2#3}{\int}$}
   \vcenter{\hbox{$#2#3$}}\kern-.5\wd0}}
\newcommand{\mint}{\mathchoice
   {\xint\displaystyle\textstyle-}
   {\xint\textstyle\scriptstyle-}
   {\xint\scriptstyle\scriptscriptstyle-}
   {\xint\scriptscriptstyle\scriptscriptstyle-}
   \!\int}

\makeatletter \catcode`@=11
\newbox\tr@tto
\setbox\tr@tto=\hbox{{\count0=0\dimen0=-,9pt\dimen1=1,1pt\loop\ifnum
    \count0<11 \advance \count0 by1 \vrule width.51pt height\dimen1
    depth\dimen0\kern-0.17pt\advance\dimen0 by-0.05pt\advance\dimen1
    by0.1pt\repeat \loop\ifnum\count0<21\advance \count0 by1 \vrule
    width.6pt height\dimen1 depth\dimen0\kern-0.2pt \advance\dimen0
    by-0.1pt\advance\dimen1 by 0.05pt\repeat}}
\def\medint{\displaystyle\copy\tr@tto\kern-10.4pt\int}
\numberwithin{equation}{section}

\allowdisplaybreaks

\begin{document}

\title[Sobolev regularity of the symmetric gradient  of solutions to  $\phi$-Laplacian systems]{Sobolev regularity of the symmetric gradient  of solutions to a class of $\phi$-Laplacian systems }

\author{ Flavia Giannetti - Antonia Passarelli di Napoli } 
\address{Flavia Giannetti, Dipartimento di Matematica e Applicazioni "R.
Caccioppoli" \\ Universit\`{a} degli Studi di Napoli ``Federico
II", via Cintia - 80126 Napoli, Italy}
\email{flavia.giannetti@unina.it}

\address{Antonia Passarelli di Napoli, Dipartimento di Matematica e Applicazioni "R.
Caccioppoli" 
\\
Universit\`{a} degli Studi di Napoli ``Federico
II", via Cintia - 80126 Napoli, Italy}
\email{antpassa@unina.it}

\bigskip

\maketitle

\medskip
\begin{center}
{\footnotesize{\textit{This paper is dedicated to Gioconda Moscariello  on the
			occasion of her 70th birthday.}}}
\end{center}

\begin{abstract}
 {The paper deals with the second order regularity properties of the  weak solutions $u\in W^{1,\phi}(\Omega, \real^n)$ } of  systems of the form
  \begin{equation*}\label{equareg}
  	-\dive A(x,\E u)=f,
  \end{equation*}
   in a bounded domain $\Omega\subset\R^n$, $n>2$,
  where the  operator $ A(x,P)$  is Lipschitz continuous with respect to the $x$-variable and satisfies growth conditions  with respect to the second variable  expressed through a Young function $\Phi$. We prove the Sobolev regularity  of a function of the symmetric gradient $\E u$ that takes into account the nonlinear growth of the operator $A(x,P)$, {assuming that the force term $f$  belongs to a suitable Orlicz-Sobolev space. {The main result is achieved through some uniform higher differentiability estimates for solutions to a class of
approximating problems, constructed adding singular higher order perturbations to the  system.}}
\end{abstract}

\medskip

\noindent {\footnotesize {\bf AMS Classifications:} 
35J60; 35B65; 35Q35.}

\noindent {\footnotesize {\bf Key words.} {\it 
    Higher differentiability,  Orlicz growth, Symmetric gradient}}

\bigskip
\bigskip
\bigskip

\tableofcontents

\section{Introduction}
In this paper, we  study the regularity of  weak solutions to  systems of the type 
\begin{equation}\label{equareg}
    -\dive A(x,\E u)=f,\qquad\qquad \text{in}\,\,\Omega
  \end{equation}
 where $\Omega\subset\R^n$,   $n\ge2$, is a bounded domain, $\E u$ denotes the symmetric part of the gradient $Du $, i.e.
  \begin{equation*}
  	\E u:=\tfrac12\big(Du+(Du)^T\big),
  \end{equation*} and the  function
  $f:\Omega\to\R^n$ is a given measurable inhomogeneity. 
  
  \noindent The mapping
 $A:\Omega\times\Sym\to\Sym$, where $\Sym$ denotes the space of symmetric $n\times n$ matrices,   will be assumed continuous in $x$ and satisfying some ellipticity and growth conditions  with respect to the second variable   expressed through a Young function. More precisely,
  we shall assume that  $P\to A(x,P)\in C^1(\Sym\setminus\{0\})\cap C^0(\Sym)$ and that there exist positive constants $\nu, L_1,L_2, K$ such that
  \begin{align}
   \label{ip1}
  \big\langle A(x,P)-A(x,Q),P-Q\big\rangle
  &\ge \nu\phi''(|P|+|Q|)|P-Q|^2\\
  \label{ip2}
  |A(x,P)-A(x,Q)|&\le L_1\,\phi''(|P|+|Q|)|P-Q|,\\
    \label{ip2reg1}
   |A(x,P)|&\le L_2\,\phi'(|P|),\\
  \label{iplip0}
  |A(x,P)-A(y,P)|
  &\le
  K|x-y| \phi'(|P|),
 \end{align}
 for a.e. $x\in\Omega$ and every $P,Q\in\R^{n\times n}_{\mathrm{sym}}$. Here  $\phi\in C^2([0, \infty))$ is a Young function satisfying the $\Delta_2$ and $\nabla_2$ conditions that can be formulated introducing the so called Simonenko indices of $\phi'$ shifted by $1$, i.e.
\begin{align}
    \label{indices}
    i_\phi= \inf_{t>0} \frac{t\phi''(t)}{\phi'(t)} +1 \quad \text{and} \quad s_\phi= \sup_{t>0} \frac{t\phi''(t)}{\phi'(t)} +1.
\end{align} and requiring that
 \begin{align}
    \label{indices3}
    1<i_\phi \leq s_\phi <\infty.
\end{align}

Note that $\phi(t)$  is not necessarily a power function and therefore the operator $A(x,P)$ can be strongly nonlinear in the second variable.

Systems such as in \eqref{equareg}  naturally  arise  in mathematical models  describing numerous physical phenomena. One of the most relevant example is given by  systems describing the motion of  non-Newtonian fluids, which are 
incompressible, {subject to constant pressure} and for which the relation between  the deformation velocity $\E u$ and the deviatoric tensor stress $A(x,P)$
is nonlinear. 
\\
Moreover, it happens to run into  systems as in \eqref{equareg} when studying classical theories of plasticity and nonlinear elasticity. Here  we deal with higher differentiability properties of $\E u$  for weak solutions $u$ to the system \eqref{equareg}. It is worth pointing out that,
recently, many papers have been devoted to the study of this kind of regularity  of $\E u$. 
\\
For example, we  mention \cite{GPS2} and \cite{Cavagnoli},  where  $p$-Navier-Stokes  and $p$-Stokes systems have been examined, respectively  for $p\geq 2$ and  for $1<p< 2$, and \cite{GPS, CGPS}  where Navier-Stokes systems with an Orlicz growth of the type considered here, possibly of power type $\phi(t)=\frac{t^p}{p}$ for any $p\in (1,\infty)$, are studied. Interior and global regularity for nonlinear systems of the type
\begin{equation*}
    -\dive A(\E u)=f,
  \end{equation*}
  have been the subject of \cite{Ser} and \cite{BeDiening}, respectively under
   $p$-growth, $1<p<2$, and  $\phi$-growth.

The nonlinear nature of the problem implies that one cannot expect, in general, 
the existence of the second weak derivatives. However, the extra differentiability can be expressed through  
a  nonlinear function of the symmetric gradient that takes into account  the precise growth of  the operator $A(x,P)$.   Actually,  already in the case $\Delta_p u=0$,  the quantity $V(Du):=|Du|^{\frac{p-2}{2}}Du$ plays a key role in the regularity theory of the gradient of  weak solutions (see \cite{Giusti} and the references therein). \\
 Accordingly, here we shall express the regularity of $\E u$ through   the function $V:\Sym\to \Sym$ associated to $\phi$ defined  by
\begin{equation}
	\label{eq:V-def}
	V(P):=
	\left\{
	\begin{array}{ll}
		\sqrt{\frac{\phi'(|P|)}{|P|}}\,P &\mbox{ if $P\neq0$}\\[1ex]
		0 &\mbox{ if $P=0$},
	\end{array}
\right.
\end{equation}
introduced in \cite{DieEtt08} and since then widely used in the context of Orlicz growth.
\\
 Most of the above mentioned papers deal with an inhomogeneity $f$  that is not weakly differentiable and this, in general, results in a higher differentiability for the solutions of fractional order.
Indeed,  the  higher differentiability of integer order cannot be expected already for  solutions to the $p$-harmonic type equations, if $p>2$,  unless some differentiability even of fractional order is assumed on the datum $f$, while for $1<p<2$, the situation  changes (see \cite{BS, BeDiening} and \cite{CGPdN, Cavagnoli}). It is worth pointing out that if one wishes to express the higher differentiability of the symmetric gradient through the operator itself, it is sufficient to assume the datum $f\in L^2$  (see for example \cite{CianchiMazya2018}).  
\\
Since we deal with a general  Young function $\phi$, we shall assume $f$ differentiable, $f\in W^{1,\phi^*}$ with $\phi^*$ conjugate of $\phi$ to be precise, even if such assumption can be more restrictive than necessary for some choices of $\phi$.
\\
Related results on the $W^{1,2}_{\rm loc}(\Omega, \mathbb R^{n\times n})$ regularity of $V(\E u)$ for local solutions to  Stokes systems, under  various growth conditions on the differential operator and different integrability assumptions of the right-hand side,  can be  found in the papers \cites{Breit, BrFu, DiKa, Na}. However, all of those contributions  concern operators independent of the space variable.

\noindent To our knowledge, no higher differentiability results of integer order for solutions of  systems as in \eqref{equareg} are available in case of operators depending on the space variable and with Orlicz growth and our aim  here is to fill this gap. Our main result states as follows
\begin{thm}\label{main}
   Let $A:\Omega\times\Sym\to\Sym$ be a mapping that is in $C^1(\Sym\setminus\{0\})\cap C^0(\Sym)$ with respect to the second variable and satisfying
the properties~\eqref{ip1}  \eqref{ip2}, \eqref{ip2reg1} and \eqref{iplip0}. Assume that \eqref{indices3} is in force and that $f\in W^{1,\phi^\star}_{\mathrm{loc}}(\Omega, \mathbb{R}^n)$.  If $u\in W^{1,\phi}(\Omega, \real^n)$ is a weak solution to \eqref{equareg}, then 
  %\begin{itemize}
  {$V(\E u)\in W^{1,2}_{\mathrm{loc}}(\Omega,\Sym)$}
  and the following inequality holds for every pair of concentric balls $B_\rho\subset B_r\Subset \Omega$ 
\begin{eqnarray}\label{firstestimate}
   \int_{B_{\rho}}|D( V(\E u))|^2\,\dx&\le& 
      c  \int_{B_{r}}\phi(|D u|)\,\dx+ c\int_{B_{r}}\phi^*\left(|Df|\right)\,\dx
 \end{eqnarray}
 for a constant $c=c(\nu,L_1,L_2,K,i_\phi, s_\phi, n, \rho, r)$.
\end{thm}
We refer to \cite{BeDiening} for an analogous result for operator $\phi$-Laplace systems independent of the $x$-variable.

We rely on the following notion of weak solutions to systems of the type \eqref{equareg}.

\begin{defn}\label{def:weak-solution}
We call $u\in W^{1,\phi}(\Omega,\real^n)$ a weak solution of
the system \eqref{equareg} if 
\begin{equation}\label{weak-equa}
  \int_{\Omega}\langle A(x,\E
  u),\E\eta\rangle\,\dx
  =
  \int_{\Omega}f\cdot\eta\,\dx
\end{equation}
is satisfied for every 
%$\eta\in C_0^\infty(\Omega,\real^{n})$.
$\eta\in W^{1,\phi}_0(\Omega,\real^{n})$.
\end{defn}

  Higher differentiability results for both the full and the symmetric gradient  are usually obtained  by means of the well known difference quotient methods (  \cite{AF,  GiaquintaModica:1986, Giusti}). Among the numerous  more recent contributions obtained under different growth conditions on the differential operator and on the integrability of the right hand side, we mention \cite{Ambro, AGPdN, Cellina, CKP1, GM, GriRus, MP1, MP2}.
  \\
Our proof is based on the construction of a suitable family of  approximating problems, whose solutions are smooth enough so that their second derivatives can be used as test function. 
 A central idea of our approach
is approximating the original problem adding  higher order perturbations. 
\\
It is worth noting that, in this respect,  our reasoning have features in common with \cite{CKP} and with \cite{DiKa}. 
Next, we shall establish uniform higher differentiability estimates
for the symmetric gradient of the solutions of the auxiliary problems, mainly through the use of an identity that relates the derivatives of the symmetric gradient of a $C^2$ function with its second derivatives. It seems that, avoiding the use of the difference quotient method has the advantage of making the
ensuing calculations easier, since we deal with smooth solutions.
The conclusion will follow by  proving that the a priori higher differentiability estimates are preserved in passing to the limit.
\\
The paper is organized as follows.  In  Section 2 we shall recall the notion and some properties of Young functions that will be needed in our proofs. In particular, we shall state a Korn-type inequality, implying the equivalence of the norm of $\E u$ and of $Du$ and a modular Poincar\'e inequality in the context of Orlicz spaces.  The approximation argument will be introduced in  Section 3, which is devoted to  the proof of the uniform a priori estimate. Finally, the proof of the main result will be the content of Section 4.

 \section{Preliminary results}
\subsection{Notation and elementary lemmas}

\noindent
We write $B_\rho(x_o)\subset\R^n$ for the open ball of radius $\rho>0$ and center
$x_o\in\R^n$.
The mean value of a function $f\in L^1(B_\rho(x_o),\R^N)$ will be
denoted by 
\begin{equation*}
  (f)_{x_o,\rho}:=\mint_{B_\rho(x_o)}f(x)\,\dx.
\end{equation*}
If the center of the ball is clear from the context, we shall
 write $(f)_\rho$ instead of $(f)_{x_o,\rho}$.
For the standard scalar product on the space
$\R^{n\times n}$ of $n\times n$ matrices, we shall use the notation
$\langle\cdot,\cdot\rangle$, while the
Euclidean scalar product on $\R^n$ will be denoted by "$\cdot$".

We will write $c$ for a general constant that may vary on different occasions, even within the
same line of estimates. Relevant dependencies on parameters and special constants will be suitably emphasized using
parentheses or subscripts.
For functions $f$ and $g$ we write $f\eqsim g$ if there are positive
constants $c_1,c_2$ with $c_1 f\le g\le c_2f$. 
\bigskip

\subsection{Interlude on Young Functions}
\label{sec:interl-orlicz-funct}

Here we introduce a
few auxiliary functions and some basic properties of {Young} functions, that will be used in what follows. 

A function {$\phi:[0,\infty)\to [0,\infty)$ is called a
Young function if and only if it is a left-continuous convex function  that vanishes at $0$ and is not constant. 
It is said to
satisfy the $\Delta_2$-condition if and only if there is a constant
$c>1$ such that $\phi(2t)\le c\,\phi(t)$ for every $t>0$. The smallest
constant $c$ with this property can be denoted by $\Delta_2(\phi)$.

The conjugate of a {Young} function~$\phi$ is defined as
\begin{align*}
  \phi^*(t) := \sup_{s \geq 0} \{ st - \phi(s) \}, \quad t \geq 0.
\end{align*}
One has that also $ \phi^*$ is  a Young function and it holds
\begin{equation}
  \label{Young-psi}
  st\le \phi(s)+ \phi^*(t)
  \qquad\mbox{for $s,t\ge0$}.
\end{equation}
For $i_\phi$ and $s_\phi$  defined in \eqref{indices},  one can verify that
\begin{align}
    \label{indices1}
    i_\phi\leq    \frac{t\phi'(t)}{\phi(t)}  \leq  s_\phi \quad \text{for $t >0$}
\end{align}
and that 
\begin{align}
    \label{indices2}
    s_\phi '\leq    \frac{t{ \phi^*}'(t)}{\phi^*(t)}  \leq  i_\phi ' \quad \text{for $t >0$,}
\end{align}
where  $i_\phi'$ and $s_\phi'$ denote  the H\"older conjugates of $i_\phi$
and $s_\phi$.
\\ From \eqref{indices1} and  \eqref{indices2}, one can deduce respectively that
 \begin{align}
   \min\{ \lambda^{i_\phi}, \lambda^{s_\phi}\} \phi(t)
    &\le \phi(\lambda t)
      \le\max\{ \lambda^{i_\phi}, \lambda^{s_\phi}\}\phi(t) \quad \text{for $t, \lambda \geq 0$}
      \label{homogeneity-phi}
      \end{align}
      \begin{align}
      \min\{ \lambda^{i_\phi'}, \lambda^{s_\phi'}\}\phi^*(t)
    &\le  \phi^*(\lambda t)
      \le \max\{ \lambda^{i_\phi'}, \lambda^{s_\phi'}\} \phi^*(t) \quad \text{for $t, \lambda \geq 0$.}
      \label{homogeneity-phi-star}
  \end{align}
  \\ Moreover, the following useful inequality holds
\begin{align}
    \label{further}
    \phi^* (\phi'(t))
  \le
   \phi(2t) \qquad
  \text{for $t>0$.}
\end{align}
We shall take advantage from the use of the  shifted  Young function. Such a notion was first given in~\cite{DieEtt08}, here we appeal to the slight variant
of~\cite[Appendix~B]{DieForTomWan19}.
\color{black}
We define the shifted Young functions $\phi_a(t)$ for $a \geq 0$ 
and $t\in \R$ by
\begin{align}
  \phi_a(t) := \int_0^t \frac{\phi'(a \vee s)}{a \vee s} s \, ds,
\end{align}
where $s_1 \vee s_2 := \max \{ s_1, s_2 \}$.
 The index~$a$ is called
the shift.  Obviously, $\phi_0= \phi$. Moreover, if
$a \eqsim b$, then $\phi_a(t) \eqsim \phi_b(t)$. Note that in the case $\phi(t):=\frac1p t^p$, the shifted Young functions satisfy
\begin{align*}
  \label{eq:phi_a-approx}
  \begin{aligned}
    \phi_a(t) &\eqsim (a \vee t)^{p-2} t^2,
    \\
    \phi_a'(t) &\eqsim (a \vee t)^{p-2} t,
  \end{aligned}
\end{align*}
with constants only depending on $i_\phi$.

 \noindent Observe that  for $\phi \in C^2(0, \infty)$  one has that  $\phi_a\in C^2((0, \infty) \setminus \{a\})$ and  the following identity holds
\begin{equation*}
  \frac{t\phi_a''(t)}{\phi_a'(t)}
  =
  \left\{
  \begin{array}{cl}
    1,&\mbox{ if $0\le t<a$,}\\[0.8ex]
    \frac{t\phi''(t)}{\phi'(t)},&\mbox{ if $t>a$.}
  \end{array}
  \right.
\end{equation*}
Consequently, \eqref{indices} implies that
\begin{equation}
  \label{phi-assumption-shifted}
  \min\{1,i_{\phi}-1\}\le \frac{t\phi_a''(t)}{\phi_a'(t)}\le \max\{1,s_{\phi}-1\}.
\end{equation}
Moreover, for $\phi$ satisfying \eqref{indices3}, it holds 
\begin{align}
    \min\{\lambda^{i_\phi},\lambda^2\} \phi_a(t)
    &\le \phi_a(\lambda t)
      \le \max\{\lambda^{s_\phi},\lambda^2\}\phi_a(t) \quad \text{for $\lambda, t \geq 0$}
      \label{homogeneity-phi-shifted}.
      \end{align}
       Hence, by \eqref{indices2}, we get
      \begin{align} 
   \min\{\lambda^{s_\phi'},\lambda^2\} \phi_a^*(t)
    &\le \phi_a^*(\lambda t)
      \le\max\{\lambda^{i_\phi'},\lambda^2\}\phi_a^*(t) \quad \text{for $\lambda, t \geq 0$.}
      \label{homogeneity-phi-star-shifted}
  \end{align}
 and, recalling the convexity of $\phi$, we obtain 
\begin{equation}
  \label{sub-additivity-phi}
  \phi_a(s+t)
  \le
  \max\{2^{s_\phi},4\}\phi_a\Big(\frac{s+t}2\Big)
   \le
  \max\{2^{s_\phi-1},2\}\big(\phi_a(s)+\phi_a(t)\big)
\end{equation} 
      Equations \eqref{Young-psi},
   \eqref{homogeneity-phi-shifted},
and~\eqref{homogeneity-phi-star-shifted} ensure  that, for
every $\delta>0$ there exists~$c_\delta=c_\delta(\delta,i_\phi,s_\phi)\geq 1$ such
that 
\begin{align}
  \label{eq:young}
  \begin{aligned}
    s\,t &\leq \delta\,\phi_{a}(t) + c_\delta\, \phi_a^*(s) \qquad \text{for $s,t,a\geq 0$.}
  \end{aligned}
\end{align}

{Since $\phi_a'$ is nondecreasing and $\phi_a$ satisfies the
$\Delta_2$-condition, we get} 
\begin{equation}
  \label{bound-phi-prime}
  s\phi_a'(s)
  \le
  \phi_a(2s)-\phi_a(s)
  \le
  c(q) \phi_a(s)
\end{equation}
while, a direct computation and estimate \eqref{bound-phi-prime}  imply
\begin{equation}\label{bound-phi-star}
  \phi_a^\ast(\phi_a'(s))
  =
  s\phi_a'(s)-\phi_a(s)
  \le
  c(q) \phi_a(s),
\end{equation}
{for any $s\ge0$.}

{Finally, the Young type inequalities \eqref{eq:young}  with $s$
replaced by $\phi_a'(s)$ together with \eqref{bound-phi-star} yield the estimates}
\begin{align}
  \label{eq:young2}
  \begin{aligned}
    \phi_a'(s)\,t &\leq \delta\,\phi_a(t) + c_\delta\, \phi_a(s),
    \\
    \phi_a'(s)\,t &\leq c_\delta\,\phi_{a}(t)+\delta\, \phi_a(s),
  \end{aligned}
\end{align}
{for $s,t,a \geq 0$, with constants $c_\delta=c_\delta(\delta,p,q)$.}

\medskip
Next Lemma yields  equivalent representations of the shifted
Young functions we shall apply in the sequel.
\begin{lem}[{\cite[Lemma 40]{DieForTomWan19}}]\label{lem:shifted-N-functions}
  Let $\phi:[0,\infty)\to[0,\infty)$ be a function as
  in~\eqref{indices1}. Then, for every $P,Q\in\R^{n\times
    n}$, we have
  \begin{equation*}
    \begin{aligned}
      \phi'_{|Q|}(|P-Q|)
      &\eqsim
      \frac{\phi'(|P|\vee|Q|)}{|P|\vee|Q|}|P-Q|,\\[0.8ex]
      \phi_{|Q|}(|P-Q|)
      &\eqsim \frac{\phi'(|P|\vee|Q|)}{|P|\vee|Q|}|P-Q|^2,
    \end{aligned}
  \end{equation*}
  where the implicit constants depend only on {$i_{\phi}$ and $s_{\phi}$}.
\end{lem}
Also we shall make strong use of the possibility to change the shift as the following Lemma suggests.
\begin{lem}[{ \cite[Corollary~44]{DieForTomWan19}}]
  \label{lem:change_of_shift}
  For~$\delta>0$ there exists~$c_\delta=c_\delta(\delta,i_\phi,s_\phi)$ such that for
  all~$P,Q \in \R^{n\times n}$ there holds
  \begin{align*}
    \phi_{|P|}(t) & \leq c_{\delta} \phi_{|Q|}(t) + \delta\,
                    |{V(P)-V(Q)}|^2,
    \\
    (\phi_{|P|})^*(t) & \leq c_{\delta} (\phi_{|Q|})^*(t) + \delta\,
                        |{V(P)-V(Q)}|^2.
    \\
  \end{align*}
\end{lem}
In particular, the choice $P=0$ allows to add a shift in the form
\begin{equation}
  \label{add-shift}
  \phi(t)
  \le
  c_\delta\phi_{|Q|}(t)+\delta|V(Q)|^2
  \le
  c_\delta\phi_{|Q|}(t)+\delta\phi(|Q|)
\end{equation}
for every $Q\in\R^{n\times n}$, where the
last estimate follows from the equivalence~\eqref{seconda} below.

\bigskip
For further needs, we note that the assumption  $P\to A(x,P)\in C^1(\Sym\setminus\{0\})\cap C^0(\Sym)$, together with \eqref{ip1}  and \eqref{iplip0}, implies that
\begin{align}
   \label{ip1reg}
  \big\langle D_PA(x,P)Q,Q\big\rangle
  &\ge c(\nu)\phi''(|P|)|Q|^2\\
  %\label{ip2reg}
 %|D_PA(x,P)|&\le c(L_1)\phi''(|P|),\\
  \label{iplip}
  |D_xA(x,P)|
  &\le
  K\phi'(|P|),
 \end{align}

\noindent The next Lemma summarizes the   relation between the operator $A$, the auxiliary function $V$
introduced in~\eqref{eq:V-def} and the shifted versions of~$\phi$.
{The proof can be found, for instance, in \cite[Lemma 3]{DieEtt08} or
  \cite[Lemma 41]{DieForTomWan19}.}

\begin{lem}
  \label{lem:hammer}
 Let $A:\Omega\times\Sym\to\Sym$ be a Carath\'eodory function with
  properties~\eqref{ip1reg} and \eqref{ip2reg1},
  where~\eqref{indices3} is in force.
  Then, for all~$P,Q \in \Sym$ and a.e. $x,y\in\Omega$ there holds
  \begin{align}\label{a-coercivity}
     \langle D_{P}A(x,P)Q,Q\rangle
    &\ge  |D_PV(P)|^2 {\eqsim \phi''(|P|)|Q|^2},\\[0.8ex]
    \label{seconda}
    \langle A(x,Q), Q \rangle
    &\eqsim |{V(Q)}|^2
    \eqsim
    \phi_{|{Q}|}(|{Q}|) \eqsim \phi(|{Q}|)
  \end{align}
  with constants that depend only on ${i_{\phi}, s_{\phi}},\nu$, and $L$.
  {If in addition, the assumption~\eqref{ip2reg1} is satisfied, then we have}
  \begin{align}
    \label{a-Lipschitz}
    |{A(x,P)-A(x,Q)}| &\le c \,\phi_{|{Q}|}'(|{P-Q}|),\\
    \big\langle A(x,P)-A(x,Q), P-Q \big\rangle
    &\eqsim \phi_{|{Q}|} \left( |{P-Q}| \right)
    \eqsim |{V(P)-V(Q)}|^2.\label{a-coercivity-2}
  \end{align}
\end{lem}

\bigskip

%****************************************************
\subsection{Orlicz spaces}
%****************************************************

Let $\phi$ be a Young function. The  Orlicz space $L^\phi(\Omega)$ is defined as 
\begin{align*}
    \label{orlicz}
    L^\phi (\Omega)= \bigg\{u: \Omega \to \mathbb R^n: \exists \lambda >0\,\,\text{s.t.} \,\, \int_\Omega\phi\bigg(\frac{|u|}{\lambda}\bigg)\, \d x <\infty\bigg\}
\end{align*}
and it becomes a Banach space with the Luxemburg norm
\begin{equation*}
  \|u\|_{L^\phi(\Omega,\R^n)}:=\inf\bigg\{\lambda>0\colon \int_\Omega
  \phi\bigg(\frac{|u(x)|}{\lambda}\bigg)\dx\le 1\bigg\}.
\end{equation*}
The Orlicz-Sobolev space $W^{1,\phi}(\Omega,\R^n)$ is defined as the space of
functions $u\in L^\phi(\Omega,\R^n)$ that are weakly differentiable
with $Du\in  L^\phi(\Omega,\R^{n\times n})$. It is  equipped with the
norm
$$\|u\|_{W^{1,\phi}(\Omega,\R^n)}:=\|u\|_{L^{\phi}(\Omega,\R^n)}+\|Du\|_{L^{\phi}(\Omega,\R^{n\times n})}.$$ 
 We refer to the space $W^{1,\phi}_0(\Omega,\R^n)$ 
of functions in $W^{1,\phi} (\Omega, \mathbb R^n)$ vanishing in $\partial \Omega$ as the subspace of functions in
$ W^{1,\phi} (\Omega, \mathbb R^n)$  whose  extension  by $0$ to $\mathbb R^n\setminus \Omega$ is weakly differentiable in $\mathbb R^n $.
\iffalse
Finally, we recall the definition of the local Besov space  $B^{\sigma, \phi, \infty}_{\rm loc}(\Omega, \mathbb R^N)$ 
as 
\begin{align}\label{besov}
B^{\sigma, \phi, \infty}_{\rm loc}(\Omega, \mathbb R^N)= \bigg\{v\in L^\phi_{\rm loc}(\Omega, \mathbb R^N): \forall\,B_R\Subset  \Omega\,\exists\, \lambda, \delta>0 \, \text{s.t.}\sup_{|h|<\delta }    \int_{B_R}\phi\bigg(\bigg|\frac{\tau_hv}{\lambda |h|^\sigma}\bigg|\bigg)\, \d x < \infty\bigg\}.
\end{align}
\fi
\noindent We conclude the Section recalling a  Korn-type inequality in the Orlicz setting (see
\cite[Thms. 6.10, 6.13]{DRS:2010}) and a modular Poincar\'e inequality,
whose proof can be found for example in  \cite{GiaPasCC,FS90}  or in \cite{Tal,BaLe} for  more general domains. 
\begin{lem}\label{Lemma-Korn}
  Let $\phi$ be a Young function satisfying \eqref{indices3}  and let $B_\rho(x_o)$ be a ball in
  $\R^n$. Then, for any $u\in
  W^{1,\phi}_0(B_\rho(x_o),\R^n)$ we have the estimate
  \begin{equation}
    \label{korn-0}
    \int_{B_\rho(x_o)}\phi(|Du|)\,\dx
    \le
    c\int_{B_\rho(x_o)}\phi(|\E u|)\dx,
  \end{equation}
  where the constant $c$ depends only on $n$, $\Delta_2(\phi)$,
  and $\Delta_2(\phi^\ast)$.
\end{lem}

\begin{prop}\label{AApoinc}
Let $\phi$ be a Young function satisfying \eqref{indices3} and let $B_R(x_o)$ be a ball in
  $\R^n$. Let $u\in W^{1,\phi}_0(B_R)$,
then there exists a positive constant  $C=C(i_\phi,s_\phi, n, R)$  such that
\begin{equation}\label{Apoinc}
 \int_{B_R} \phi(|u|)\, \dx \leq C
\,  \int_{B_R}\phi(|Du|) \dx.
\end{equation}

\end{prop}

\bigskip

\section{The approximation}
We argue as in \cite{CKP,DiKa}, by introducing a sequence of problems with a singular higher order perturbation. More precisely, fix a ball $B_R\Subset \Omega$ and let  $k\in \mathbb{N}$ be large enough to have
$$W^{k,2}(B_R,\R^n)\hookrightarrow C^2(B_R,\R^n)$$
and consider the solutions $u_{k,\varepsilon}\in W^{k,2}(B_R, \R^n)$ to the  following family of problems
\begin{eqnarray}\label{probapp}
\begin{cases}\displaystyle{\tilde\varepsilon\int_{B_R} \langle D^{k}u_{k,\varepsilon},D^{k}\psi\rangle\,\dx=-\int_{B_R} \langle A(x,\E u_{k,\varepsilon}),\E \psi\rangle\,\dx+\int_{B_R}\langle f_\varepsilon,\psi\rangle\,dx},\cr\cr
u_{k,\varepsilon}= {u}_\varepsilon \qquad\qquad\text{on}\,\, \partial B_R
\end{cases}
\end{eqnarray}
 for every $\psi\in W^{k,2}_0(B_R,\R^n)$ 
 and where ${u}_\varepsilon$ and $f_\varepsilon$ represent  standard mollification of $u$  and $f$, respectively. Moreover, we choose $\tilde \varepsilon=\tilde \varepsilon(\varepsilon,u)$ such that\begin{equation}\label{sceltaepsilon}\tilde \varepsilon\int_{B_R}|D^k {u}_\varepsilon|^2\,dx\to 0\end{equation}
Note that, by our choice of $k$, we have  $u_{k,\varepsilon}\in C^2(B_R,\R^n).$
The existence and the uniqueness of $u_{k,\varepsilon}$ can be easily proven arguing as in \cite[Lemma 3.3]{DiKa} as well as their higher differentiability, i.e. \begin{equation}
    \label{highdiffreg}u_{k,\varepsilon}\in W^{k+1,2}_{\loc}(B_R,\R^n)
\end{equation} (see \cite[Lemma 3.4]{DiKa})
Therefore, we can use integration by parts in \eqref{probapp} choosing $\psi=D_j\varphi$ with $\varphi\in W^{k+1,2}_0(B_R,\R^n)$ and find that $u_{k,\varepsilon}$ solves the following identity} 
  \begin{equation}\label{equa-parts}
       -\tilde\varepsilon\int_{B_R} \langle D_jD^{k}u_{k,\varepsilon},D^{k}\varphi\rangle\,\dx=\int_{B_R} \langle D_j\big(A(x,\E u_{k,\varepsilon})\big),\E \varphi\rangle\,\dx-\int_{B_R}\langle D_jf_\varepsilon,\varphi\rangle\,dx, \end{equation}
     for every $\varphi\in W^{k,2}_0(B_{r},\R^n)$ 
 and for every $j=1,\dots,n$.  
\\
{The main aim of this section is to prove a uniform higher differentiability estimate for the solutions to the problems \eqref{probapp}. 
More precisely, we have 
\begin{thm} Let $u_{k,\varepsilon}\in \tilde u_\varepsilon+W^{k,2}_0(B_R,\R^n)$  be a solution to \eqref{probapp}. If $f\in W^{1,\phi^*}_{\mathrm{loc}}(\Omega,\mathbb{R}^n)$, then the following estimate holds 
\begin{eqnarray}\label{estdersecenunciato}
   \int_{B_{\rho}}|D( V(\E u_{k,\varepsilon}))|^2\dx&\le& 
     c\left(1+\frac{k^2}{( r- \rho)^2}\right)\int_{B_{ r}}\phi\left(|Du_{k,\varepsilon}|\right)\,\dx+c\int_{B_{ r}}\phi^*(|Df_\varepsilon|)\dx\cr\cr
    &&+\frac{c_k\tilde\varepsilon}{( r-\rho)^{2k}}\int_{B_{ r}} \left(\sum_{i=0}^{k-1}|D^i Du_{k,\varepsilon}|^2\right)\,\dx,
 \end{eqnarray}
 for every pair of concentric balls $B_{\rho}\subset B_{ r}\Subset B_R$, where   the constant $c=c(K,L_1,L_2,i_{\phi}, s_{\phi},n )$ is independent of $k$ and $c_k$ also depends on $k$.
\end{thm}}
\begin{proof}Fix radii $0<\rho\le r <R $ and a smooth cut-off function $\eta$ satisfying $\mathbf{1}_{B_{\rho}} \le \eta\le \mathbf{1}_{B_{r}}$ and $|D^i\eta|\le \frac{c}{( r-\rho)^i}$ for each $i\in \mathbb{N}$. Without loss of generality, in what follows, we assume $0<R\le 1$.
  By property~\eqref{highdiffreg} of $u_{k,\varepsilon}$, we have
  $\varphi=\eta^{2k}D_ju_{k,\varepsilon}\in W^{k,2}_0(B_{r},\R^n) $, which makes it admissible as
  a test function in \eqref{equa-parts}. With this choice, for each direction $1\le j\le n$ we get
\begin{eqnarray}\label{inizio}
  0&=&-\tilde\varepsilon\int_{B_r} \langle D_jD^{k}u_{k,\varepsilon},D^{k}(\eta^{2k}D_j u_{k,\varepsilon})\rangle\,\dx- \int_{B_r} \eta^{2k}\langle D_j\big(A(x,\E u_{k,\varepsilon})\big),\E(D_j u_{k,\varepsilon})\rangle\,\dx\cr\cr
  &&-2k\int_{B_r} \eta^{2k-1}\langle D_j\big(A(x,\E u_{k,\varepsilon})\big),D_j u_{k,\varepsilon}\nabla\eta\rangle\,\dx+\int_{B_r} \eta^{2k}\langle D_jf_\varepsilon, D_j u_{k,\varepsilon}\rangle\,\dx\cr\cr
  &=:&\mathrm{J}+\mathrm{I}+\mathrm{II}+\mathrm{III}.
\end{eqnarray}
In order to estimate $\mathrm{J}$, we use   the chain rule, Cauchy–Schwarz inequality, the properties of $\eta$ so that  
\begin{eqnarray}\label{J}
    \mathrm{J} &\le & -\tilde\varepsilon\int_{B_r} \eta^{2k}|D_jD^ku_{k,\varepsilon}|^2\,\dx+\frac{c_k\tilde\varepsilon}{(r-\rho)^k}\int_{B_r} \eta^{k}|D_jD^ku_{k,\varepsilon}|\sum_{i=0}^{k-1}|D^i D_ju_{k,\varepsilon}|\,\dx\cr\cr 
    &\le & -\frac{2\tilde\varepsilon}{3}\int_{B_r} \eta^{2k}|D_jD^ku_{k,\varepsilon}|^2\,dx+\frac{c_k\tilde\varepsilon}{(r-\rho)^{2k}}\int_{B_r} \left(\sum_{i=0}^{k-1}|D^i D_ju_{k,\varepsilon}|\right)^2\,\dx\cr\cr
    &\le & -\frac{2\tilde\varepsilon}{3}\int_{B_r} \eta^{2k}|D_jD^ku_{k,\varepsilon}|^2\,\dx+\frac{c_k\tilde\varepsilon}{(r-\rho)^{2k}}\int_{B_r} \sum_{i=0}^{k-1}|D^i D_ju_{k,\varepsilon}|^2\,\dx,
\end{eqnarray}
where we also used $r-\rho\le 1$.

Calculating the differential of $A(x,\E u_{k,\varepsilon})$, we  write $\mathrm{I}$  as follows
\begin{eqnarray}\label{I}
    \mathrm{I}
    &&=-\int_{B_r}\eta^{2k} \langle D_{ P} A(x,\E u_{k,\varepsilon})D_j \E u_{k,\varepsilon},D_j \E  u_{k,\varepsilon}\rangle\,\dx-\int_{B_r}\eta^{2k} \langle D_{x_j}A(x,\E u_{k,\varepsilon}),D_j \E u_{k,\varepsilon}\rangle\,\dx\cr\cr&&=:-\mathrm{I}_1-\mathrm{I}_2,
\end{eqnarray}
while, an integration by parts  in $\mathrm{II}$ yields
\begin{eqnarray}\label{II}
    \mathrm{II}
    &&=2k\int_{B_r} \langle  A(x,\E u_{k,\varepsilon}),D_j (\eta^{2k-1}D_j  u_{k,\varepsilon}\nabla\eta)\rangle\,\dx\cr\cr&&=2k\int_{B_r}\eta^{2k-1} \langle  A(x,\E u_{k,\varepsilon}),D_j^2u_{k,\varepsilon}\nabla\eta\rangle\,\dx+2k\int_{B_r}\eta^{2k-1} \langle  A(x,\E u_{k,\varepsilon}),D_j  u_{k,\varepsilon}D_j(\nabla\eta)\rangle\,\dx\cr\cr
    &&+2k(2k-1)\int_{B_r}\eta^{2k-2} \langle  A(x,\E u_{k,\varepsilon}),D_j  u_{k,\varepsilon}D_j\eta\nabla\eta\rangle\,\dx\cr\cr 
    &&=:\mathrm{II}_1+\mathrm{II}_2+\mathrm{II}_3.
\end{eqnarray}
Inserting \eqref{I} and \eqref{II} in \eqref{inizio}, we get
$$ \mathrm{I}_1=-\mathrm{I}_2+\mathrm{II}_1+\mathrm{II}_2+\mathrm{II}_3+ \mathrm{III}+\mathrm{J}$$
and so, by \eqref{J}, we infer that 
\begin{eqnarray}\label{iniziale}
    \mathrm{I}_1&\le& |\mathrm{I}_2|+|\mathrm{II}_1|+|\mathrm{II}_2|+|\mathrm{II}_3|+ |\mathrm{III}|-\frac{2\tilde\varepsilon}{3}\int_{B_r} \eta^{2k}|D_jD^ku_{k,\varepsilon}|^2\,\dx\cr\cr 
    &&+\frac{c_k\tilde\varepsilon}{(r-\rho)^{2k}}\int_{B_r} \left(\sum_{i=0}^{k-1}|D^i D_ju_{k,\varepsilon}|^2\right)\,\dx\cr\cr 
    &\le& |\mathrm{I}_2|+|\mathrm{II}_1|+|\mathrm{II}_2|+|\mathrm{II}_3|+ |\mathrm{III}|+\frac{c_k\tilde\varepsilon}{(r-\rho)^{2k}}\int_{B_r} \left(\sum_{i=0}^{k-1}|D^i D_ju_{k,\varepsilon}|^2\right)\,\dx
\end{eqnarray}
According to~\eqref{a-coercivity}, the left-hand side of the previous inequality is bounded from below as
 \begin{equation}
    \label{lower-boundreg}
   \mathrm{I_1}
    \ge
     c(\nu)\int_{B_{r}}\eta^{2k}|D_j( V(\E u_{k,\varepsilon}))|^2\,\dx.
  \end{equation} 
  Using in turn assumption~\eqref{iplip} and Young's inequality, with the help of \eqref{indices},  the equivalence in \eqref{a-coercivity} and \eqref{indices1},
  we have
  \begin{align}\label{pre-I-2-bound}
    |\mathrm{I}_2|
    &\le
      \int_{B_r}\eta^{2k} | D_jA(x,\E u_{k,\varepsilon})|\,|D_j \E u_{k,\varepsilon}|\,\dx\\\nonumber
    &\le
    K\int_{B_{r}}\eta^{2k}\phi'(|\E u_{k,\varepsilon}|)|D_j\E u_{k,\varepsilon}|\,\dx\\\nonumber
    &=
       K\int_{B_{r}}\eta^{2k}\frac{\phi'(|\E u_{k,\varepsilon}|)}{|\E u_{k,\varepsilon}|^{\frac12}}\,|\E u_{k,\varepsilon}|^{\frac12}\,|D_j\E u_{k,\varepsilon}|\,\dx\\\nonumber
    &\le
      \sigma \int_{B_{r}}\eta^{2k}\frac{\phi'(|\E u_{k,\varepsilon}|)}{|\E u_{k,\varepsilon}|}\,|D_j\E u_{k,\varepsilon}|^2\,\dx+
      c_\sigma \int_{B_{r}}\eta^{2k}\phi'(|\E u_{k,\varepsilon}|)\,|\E u_{k,\varepsilon}|\,\dx\\\nonumber&\le
      c\sigma \int_{B_{r}}\eta^{2k}\phi''(|\E u_{k,\varepsilon}|)\,|D_j\E u_{k,\varepsilon}|^2\,\dx+
      c_\sigma \int_{B_{r}}\eta^{2k}\phi'(|\E u_{k,\varepsilon}|)\,|\E u_{k,\varepsilon}|\,\dx\\\nonumber
    &\le
      c\sigma \int_{B_{r}}\eta^{2k}|D_j( V(\E u_{k,\varepsilon}))|^2\,\dx+
      c_\sigma \int_{B_{r}}\phi(|\E u_{k,\varepsilon}|)\,\dx,
  \end{align}
  where $\sigma>0$ is a parameter that will be chosen later.
The assumption~\eqref{ip2reg1} and the properties of $\eta$ also yield 
  \begin{align*}
    |\mathrm{II}_1|
    &\le
      \frac{ck}{r-\rho}\int_{B_r}\eta^{2k-1}\phi'(|\E u_{k,\varepsilon}|)\,|D^2_j  u_{k,\varepsilon}|\,\dx.
    \end{align*}
    Since $u_{k,\varepsilon}\in C^2(B_R)$, we may use the algebraic identity 
  \begin{equation*}\label{identity}
   \frac{\partial^2 u_{k,\varepsilon}^i}{\partial x_\ell\partial x_j }=\frac{\partial}{\partial x_\ell}\E_{ij}u_{k,\varepsilon}+\frac{\partial}{\partial x_j}\E_{i\ell}u_{k,\varepsilon}-\frac{\partial}{\partial x_i}\E_{j\ell}u_{k,\varepsilon}   \end{equation*}
   in the particular case $\ell=j$,
  and deduce from the previous estimate that
  \begin{align}\label{pre-II-bound}
    |\mathrm{II}_1|
    &
      \le
    \frac{ck}{r-\rho}\int_{B_{r}}\eta^{2k-1}\phi'(|\E u_{k,\varepsilon}|)|\,|D(\E  u_{k,\varepsilon})|\,\dx\\\nonumber
    &\le
      c\sigma \int_{B_{r}}\eta^{2k}|D( V(\E u_{k,\varepsilon}))|^2\,\dx+
      \frac{c_\sigma k^2}{(r-\rho)^2}  \int_{B_{r}}\phi(|\E u_{k,\varepsilon}|)\,\dx,
  \end{align}
  where we argued as in \eqref{pre-I-2-bound}.
   Using  assumption~\eqref{ip2reg1}, the properties of $\eta$ and the inequality  \eqref{eq:young2} with $a=0$, we obtain
  \begin{align}\label{pre-I-3-bound}
    |\mathrm{II}_2|+|\mathrm{II}_3|
    &\le
      \frac{c(k)}{(r-\rho)^2}\int_{B_r} \phi'(|\E u_{k,\varepsilon}|)\,|D_j u_{k,\varepsilon}|\,\dx\\\nonumber
    &\le
      \frac{c(k)}{(r-\rho)^2}\int_{B_{r}}\,\phi(|\E u_{k,\varepsilon}|)\,\dx+
      \frac{c (k)}{(r-\rho)^2}  \int_{B_{r}}\phi(|D u_{k,\varepsilon}|)\,\dx.
    \end{align}
 In order to estimate $|\mathrm{III}|$, we use the assumption $f\in W^{1,\phi^*}_{\mathrm{loc}}(\Omega)$ and Young's inequality at  \eqref{Young-psi}, having   
 \begin{align}\label{pre-III-bound}
    |\mathrm{III}|
    \le \int_{B_{r}}|Df_\varepsilon|\cdot|D u_{k,\varepsilon}|\,\,\dx\le 
\int_{B_{r}}\phi\left(|D u_{k,\varepsilon}|\right)\,\dx+\int_{B_{r}}\phi^*(|Df_\varepsilon|)\,\dx,
    \end{align}
     where we also used that $0\le\eta\le 1$.
   
  Inserting the bounds \eqref{lower-boundreg}, \eqref{pre-I-2-bound}, \eqref{pre-II-bound}, \eqref{pre-I-3-bound} and \eqref{pre-III-bound} in \eqref{iniziale}, we get
 \begin{eqnarray*}
   \int_{B_{r}}\eta^{2k}|D_j( V(\E u_{k,\varepsilon}))|^2\,\dx&\le& c\sigma \int_{B_{r}}\eta^{2k}|D_j( V(\E u_{k,\varepsilon}))|^2\,\dx+c\sigma \int_{B_{r}}\eta^{2k}|D( V(\E u_{k,\varepsilon}))|^2\,\dx\cr\cr
&&+c_\sigma\left(1+\frac{k^2}{(r-\rho)^2}\right)\int_{B_{r}}\Big(\phi(|\E u_{k,\varepsilon}|)+\phi\left(|Du_{k,\varepsilon}|\right)\Big)\,\dx\cr\cr
    &&+c\int_{B_{r}}\phi^*(|Df_\varepsilon|)\,\dx+\frac{c_k\tilde\varepsilon}{(r-\rho)^{2k}}\int_{B_r} \left(\sum_{i=0}^{k-1}|D^i D_ju_{k,\varepsilon}|^2\right)\,\dx.
 \end{eqnarray*} 
 Summing  over $j=1,\dots,n$ the previous inequalities, we have
 \begin{eqnarray*}
   \int_{B_{r}}\eta^{2k}|D( V(\E u_{k,\varepsilon}))|^2\,\dx&\le& c\sigma \int_{B_{r}}\eta^{2k}|D( V(\E u_{k,\varepsilon}))|^2\,\dx\cr\cr &&+c_\sigma\left(1+\frac{k^2}{(r-\rho)^2}\right)\int_{B_{r}}\Big(\phi(|\E u_{k,\varepsilon}|)+\phi\left(|Du_{k,\varepsilon}|\right)\Big)\,\dx\cr\cr
    &&+c\int_{B_{r}}\phi^*(|Df_\varepsilon|)\,\dx+\frac{c_k\tilde\varepsilon}{(r-\rho)^{2k}}\int_{B_r} \left(\sum_{i=0}^{k-1}|D^i Du_{k,\varepsilon}|^2\right)\,\dx.
 \end{eqnarray*} 
 Choosing $\sigma$ sufficiently small, we can reabsorb the first integral in the right hand side by the left hand side, thus getting
  \begin{eqnarray*}
   \int_{B_{r}}\eta^{2k}|D( V(\E u_{k,\varepsilon}))|^2\,\dx&\le& 
      c\left(1+\frac{k^2}{(r-\rho)^2}\right)\int_{B_{r}}\Big(\phi(|\E u_{k,\varepsilon}|)\,\dx+\phi\left(|Du_{k,\varepsilon}|\right)\Big)\,\dx\cr\cr
    &&+c\int_{B_{r}}\phi^*(|Df_\varepsilon|)\,\dx+\frac{c_k\tilde\varepsilon}{(r-\rho)^{2k}}\int_{B_r} \left(\sum_{i=1}^{k-1}|D^i Du_{k,\varepsilon}|^2\right)\,\dx.
 \end{eqnarray*}
Recalling that $\eta\equiv 1$ on $B_\rho$, we obtain
the conclusion.
\end{proof}

\bigskip

 \section{Proof of Theorem \ref{main}}
This section is devoted to the proof of our main result. We shall prove that the uniform Sobolev estimates proven in previous section are preserved in passing to the limit. Next we will {show} that the solutions to the approximating problems converge to the solution   $u$ of \eqref{equareg} {that will consequently inherit the same property of higher differentiability}.
\begin{proof}[Proof of Theorem \ref{main}] Choosing $\varphi=u_{k,\varepsilon}- u_{\varepsilon}$ as test function in \eqref{probapp}, we have
\begin{eqnarray*}
  0&=&\tilde\varepsilon\int_{B_R} \langle D^ku_{k,\varepsilon},D^k(u_{k,\varepsilon}- u_{\varepsilon})\rangle\,\dx+ \int_{B_R} \langle A(x,\E u_{k,\varepsilon}),\E(u_{k,\varepsilon}- u_{\varepsilon})\rangle\,\dx\cr\cr
  &&-\int_{B_R} f_\varepsilon\cdot(u_{k,\varepsilon}- u_{\varepsilon})\,\dx\cr\cr
  &=&\tilde\varepsilon\int_{B_R} |D^ku_{k,\varepsilon}|^2\,\dx-\tilde\varepsilon\int_{B_s} \langle D^ku_{k,\varepsilon}, D^k u_{\varepsilon}\rangle\,\dx+ \int_{B_R} \langle A(x,\E u_{k,\varepsilon}),\E u_{k,\varepsilon}\rangle\,\dx\cr\cr
  &&- \int_{B_R} \langle A(x,\E u_{k,\varepsilon}),\E u_{\varepsilon}\rangle\,\dx-\int_{B_R} f_\varepsilon\cdot(u_{k,\varepsilon}- u_{\varepsilon})\,\dx
\end{eqnarray*}
that yields
\begin{eqnarray*}
  &&\tilde\varepsilon\int_{B_R} |D^ku_{k,\varepsilon}|^2\,\dx+ \int_{B_R} \langle A(x,\E u_{k,\varepsilon}),\E u_{k,\varepsilon}\rangle\,\dx\cr\cr
  &\le&\tilde\varepsilon\int_{B_R} | D^ku_{k,\varepsilon}|\,| D^k u_{\varepsilon}|\,\dx+ \int_{B_R} | A(x,\E u_{k,\varepsilon})|\,|\E u_{\varepsilon}|\,\dx\cr\cr
  &&+\int_{B_R} |f_\varepsilon|\,|u_{k,\varepsilon}- u_{\varepsilon}|\,\dx\cr\cr
  &\le&\frac{\tilde\varepsilon}{2}\int_{B_R} | D^ku_{k,\varepsilon}|^2\,\dx+\frac{\tilde\varepsilon}{2}\,\int_{B_R} | D^k u_{\varepsilon}|^2\,\dx\cr\cr &&+ L_2\int_{B_R}  \phi'(\E u_{k,\varepsilon})\,|\E u_{\varepsilon}|\,\dx+\int_{B_R} |f_\varepsilon|\,|u_{k,\varepsilon}- u_{\varepsilon}|\,\dx,
\end{eqnarray*}
where, in the last line, we used Young's inequality and assumption \eqref{ip2reg1}. Reabsorbing the first integral in the right hand side to the left hand side, we get
\begin{eqnarray*}
  &&\frac{\tilde\varepsilon}{2}\int_{B_R} |D^ku_{k,\varepsilon}|^2\,\dx+ \int_{B_R} \langle A(x,\E u_{k,\varepsilon}),\E u_{k,\varepsilon}\rangle\,\dx\cr\cr
  &\le&\frac{\tilde\varepsilon}{2}\,\int_{B_R} | D^k u_{\varepsilon}|^2\,\dx+ L_2\int_{B_R} \phi'(|\E u_{k,\varepsilon}|)\,|\E u_{\varepsilon}|\,\dx\cr\cr &&+\int_{B_R} |f_\varepsilon|\,|u_{k,\varepsilon}- u_{\varepsilon}|\,\dx.
\end{eqnarray*}
 Using the last equivalence in \eqref{seconda} to bound from below the second integral in the left hand side of previous estimate, we infer that
\begin{eqnarray*}
  && \frac{\tilde\varepsilon}{2}\int_{B_R} |D^ku_{k,\varepsilon}|^2\,\dx+c\int_{B_R} \phi(|\E u_{k,\varepsilon}|)\,\dx\cr\cr
  &\le&\frac{\tilde\varepsilon}{2}\,\int_{B_R} | D^k u_{\varepsilon}|^2\,\dx+ L_2\int_{B_R}  \phi'(|\E u_{k,\varepsilon}|)\,|\E u_{\varepsilon}|\,\dx\cr\cr &&+\int_{B_R} |f_\varepsilon|\,|u_{k,\varepsilon}- u_{\varepsilon}|\,\dx.
\end{eqnarray*}
 Applying the first estimate in \eqref{eq:young2} and inequality \eqref{eq:young} with  $a=0$ and using Poincar\'e inequality \eqref{Apoinc}  in the right hand side of previous estimate, we obtain that 
\begin{eqnarray*}
  && \frac{\tilde\varepsilon}{2}\int_{B_R} |D^ku_{k,\varepsilon}|^2\,\dx+c\int_{B_R} \phi(|\E u_{k,\varepsilon}|)\,\dx\cr\cr
  &\le&\frac{\tilde\varepsilon}{2}\,\int_{B_R} | D^k u_{\varepsilon}|^2\,\dx+ \sigma\int_{B_R}  \phi(|\E u_{k,\varepsilon}|)\,\dx+c_\sigma\int_{B_R} \phi(|\E u_{\varepsilon}|)\,\dx\cr\cr
  &&+\sigma\int_{B_R} \phi(|u_{k,\varepsilon}- u_{\varepsilon}|)\,\dx+c_\sigma\int_{B_R} \phi^*(|f_\varepsilon|)\,\dx\cr\cr
  &\le&\frac{\tilde\varepsilon}{2}\,\int_{B_R} | D^k u_{\varepsilon}|^2\,\dx+ \sigma\int_{B_R}  \phi(|\E u_{k,\varepsilon}|)\,\dx+c_\sigma\int_{B_R} \phi(|\E u_{\varepsilon}|)\,\dx\cr\cr &&+\sigma\int_{B_R} \phi(|Du_{k,\varepsilon}-D u_{\varepsilon})|)\,\dx+c_\sigma\int_{B_R} \phi^*(|f_\varepsilon|)\,\dx\cr\cr
&\le&\frac{\tilde\varepsilon}{2}\,\int_{B_R} | D^k u_{\varepsilon}|^2\,\dx+ \sigma\int_{B_R}  \phi(|\E u_{k,\varepsilon}|)\,\dx+c_\sigma\int_{B_R} \phi(|\E u_{\varepsilon}|)\,\dx\cr\cr &&+c\,\sigma\int_{B_R} \phi(|\E (u_{k,\varepsilon}- u_{\varepsilon})|)\,\dx+c_\sigma\int_{B_R} \phi^*(|f_\varepsilon|)\,\dx,
\end{eqnarray*}
where, in the last line, we used  Korn's inequality  \eqref{korn-0}. The use of  \eqref{sub-additivity-phi} to bound the second last term of the previous estimate then yields 
\begin{eqnarray*}
  && \frac{\tilde\varepsilon}{2}\int_{B_R} |D^ku_{k,\varepsilon}|^2\,\dx+c\int_{B_R} \phi(|\E u_{k,\varepsilon}|)\,\dx\cr\cr
&\le&\frac{\tilde\varepsilon}{2}\,\int_{B_R} | D^k u_{\varepsilon}|^2\,\dx+ c\sigma\int_{B_R}  \phi(|\E u_{k,\varepsilon}|)\,\dx+c_\sigma\int_{B_R} \phi(|\E u_{\varepsilon}|)\,\dx\cr\cr &&+c_\sigma\int_{B_R} \phi^*(|f_\varepsilon|)\,\dx.
\end{eqnarray*}
Choosing $\sigma$ sufficiently small, we are legitimate to  reabsorb the second integral in the right hand side by the left hand side and conclude that
\begin{eqnarray}\label{bound}
  && \frac{\tilde\varepsilon}{2}\int_{B_R} |D^ku_{k,\varepsilon}|^2\,\dx+c\int_{B_R} \phi(|\E u_{k,\varepsilon}|)\,\dx\le \frac{\tilde\varepsilon}{2}\,\int_{B_R} | D^k u_{\varepsilon}|^2\,\dx\cr\cr
  &&\qquad+c\int_{B_R} \phi(|\E u_{\varepsilon}|)\,\dx+c\int_{B_R} \phi^*(|f_\varepsilon|)\,\dx.
\end{eqnarray}
At this point, we recall our choice of $\tilde \varepsilon$ in \eqref{sceltaepsilon}  and observe that, since $\tilde u_\varepsilon$ and $f_\varepsilon$ are mollifiers of $u$ and $f$ respectively,   the following equalities hold
\begin{equation}\label{convmollificate}\lim_{\varepsilon\to 0}\int_{B_R}\phi(|\E u_{\varepsilon}-\E u|)\,\dx=0,\end{equation}
and
\begin{equation}\label{convmollificate1}\lim_{\varepsilon\to 0}\int_{B_R}\Big(\phi^*(|D f_{\varepsilon}-D f|)+\phi^*(| f_{\varepsilon}- f|)\Big)\,\dx=0.\end{equation} 
(see \cite{Gossez-studia}).Hence, the right hand side of \eqref{bound} can be uniformly bounded with respect to $\varepsilon$ as follows
\begin{eqnarray}\label{boundunif}
  && \frac{\tilde\varepsilon}{2}\int_{B_R} |D^ku_{k,\varepsilon}|^2\,\dx+c\int_{B_R} \phi(|\E u_{k,\varepsilon}|)\,\dx\cr\cr
  &\le& c\left(1+\int_{B_R} \phi(|\E u|)\,\dx+\int_{B_R} \phi^*(|f|)\,\dx\right).
\end{eqnarray}
Therefore, there exists $v\in {u+}W^{1,\phi}_0(B_R)$ such that
$$\E u_{k,\varepsilon}\rightharpoonup \E v \qquad\qquad \text{in}\,\, L^\phi(B_R),$$
as $\varepsilon\to 0$. On the other hand, recalling that $\{u_{k,\varepsilon}\}$ satisfy \eqref{estdersecenunciato},  estimate \eqref{boundunif}, and  the Gagliardo Nirenberg interpolation inequality, give

\begin{equation}\label{estdersecunif}
   \int_{B_{\rho}}|D( V(\E u_{k,\varepsilon}))|^2\,\dx\le 
     \tilde c\left(1+\int_{B_R}\!\!\!\! \Big({\phi(|D u|)}+\phi^*(|f|)+ \phi^*(|Df|)\Big)\,\dx\right),
 \end{equation}
where $\tilde c$ also depends on  $k,r,\rho$ but is independent of $\varepsilon$. Therefore, by \eqref{estdersecunif}, we deduce  the existence of $w\in W^{1,2}_{\mathrm{loc}}(B_R)$ such that , for $\varepsilon\to 0$, it holds
$$V(\E u_{k,\varepsilon})\rightharpoonup  w \qquad\qquad \text{in}\,\,W^{1,2}_{\mathrm{loc}}(B_R).$$
 Hence, as $\varepsilon\to 0$, we have
$$V(\E u_{k,\varepsilon})\to  w \qquad\qquad \text{in}\,\,L^{2}_{\mathrm{loc}}(B_R)$$
and, up to a subsequence,
{$$V(\E u_{k,\varepsilon})\to  w \quad \text{a.e. in}\,\, B_R.$$} Now, the uniqueness of the weak limit and  the continuity of the function $V(\xi)$ ensure that $$V(\E v)= w \qquad \text{a.e. in}\,\, B_R.$$
Therefore, by the last equivalence in  \eqref{a-coercivity-2}, it follows that
\begin{equation}\label{limv}\E u_{k,\varepsilon}\to \E v \qquad\qquad \text{strongly in}\,\, L^\phi(B_R). 
\end{equation}
Our next aim is to prove that $u=v$ a.e. in $B_R$.  Using the first equivalence in  \eqref{a-coercivity-2}, we get
\color{black}
\begin{eqnarray*}
    &&\int_{B_R}\phi_{|\E \tilde u_\varepsilon|}(|\E u_{k,\varepsilon}-\E \tilde u_\varepsilon|)\,\dx\cr\cr 
    &\le& \int_{B_R}\langle A(x,\E u_{k,\varepsilon})-A(x,\E \tilde u_\varepsilon),\E u_{k,\varepsilon}-\E \tilde u_\varepsilon\rangle\,\dx\cr\cr
	&=&\int_{B_R}\langle A(x,\E u_{k,\varepsilon}),\E u_{k,\varepsilon}-\E \tilde u_\varepsilon\rangle\,\dx-\int_{B_R}\langle A(x,\E u),\E u_{k,\varepsilon}-\E \tilde u_\varepsilon\rangle\,\dx\cr\cr
    &&-\int_{B_R}\langle A(x,\E \tilde u_\varepsilon)-A(x,\E  u),\E u_{k,\varepsilon}-\E \tilde u_\varepsilon\rangle\,\dx\cr\cr
    &=&-\tilde\varepsilon\int_{B_R} \langle D^ku_{k,\varepsilon},D^k( u_{k,\varepsilon}- \tilde u_\varepsilon)\rangle\,\dx+\int_{B_R} f_\varepsilon\cdot( u_{k,\varepsilon}- \tilde u_\varepsilon)\,\dx\cr\cr &&-\int_{B_R} f\cdot( u_{k,\varepsilon}- \tilde u_\varepsilon)\,dx-\int_{B_R}\langle A(x,\E \tilde u_\varepsilon)-A(x,\E  u),\E u_{k,\varepsilon}-\E \tilde u_\varepsilon\rangle\,\dx\cr\cr
    &=&-\tilde\varepsilon\int_{B_R}  |D^ku_{k,\varepsilon}|^2\,dx+\tilde\varepsilon\int_{B_R} \langle D^ku_{k,\varepsilon},D^k \tilde u_\varepsilon\rangle\,\dx\cr\cr &&-\int_{B_R}\langle A(x,\E \tilde u_\varepsilon)-A(x,\E  u),\E u_{k,\varepsilon}-\E \tilde u_\varepsilon\rangle\,\dx+\int_{B_R} (f_\varepsilon-f)\cdot( u_{k,\varepsilon}- \tilde u_\varepsilon)\,\dx,
\end{eqnarray*}
where, in the second last  identity, we used that $u_{k,\varepsilon}$ solves problem \eqref{probapp} and that $u$ solves equation \eqref{weak-equa}.
{Moving the first integral in the right and side to the left hand side of  previous estimate and using} in turn \eqref{a-Lipschitz},  inequalities  \eqref{eq:young2} and  \eqref{eq:young} with $a=0$, the  Poincar\'e inequality at \eqref{Apoinc} and  Korn's inequality \eqref{korn-0},  we deduce that
\begin{eqnarray*}
    &&\int_{B_R}\phi_{|\E \tilde u_\varepsilon|}(|\E u_{k,\varepsilon}-\E \tilde u_\varepsilon|)\,\dx+\tilde\varepsilon\int_{B_R}  |D^ku_{k,\varepsilon}|^2\,\dx\cr\cr 
    &\le& \int_{B_R}  |A(x,\E \tilde u_\varepsilon)-A(x,\E  u)||\E u_{k,\varepsilon}-\E \tilde u_\varepsilon|\,\dx+\tilde\varepsilon\int_{B_R} | D^ku_{k,\varepsilon}||D^k \tilde u_\varepsilon|\,\dx\cr\cr
    &&+\int_{B_R} |f_\varepsilon-f|\cdot| u_{k,\varepsilon}- \tilde u_\varepsilon|\,\dx\cr\cr
    &\le& \int_{B_R} \phi'_{|\E \tilde u_\varepsilon|}(|\E u-\E  u_{\varepsilon}|)|\E u_{k,\varepsilon}-\E \tilde u_\varepsilon|\,\dx+\tilde\varepsilon\int_{B_R} | D^ku_{k,\varepsilon}||D^k \tilde u_\varepsilon|\,\dx\cr\cr&&+\int_{B_R} |f_\varepsilon-f|\cdot| u_{k,\varepsilon}- \tilde u_\varepsilon|\,\dx\cr\cr &\le &
    c\int_{B_R} \phi_{|\E \tilde u_\varepsilon|}\left(|\E u-\E  u_{\varepsilon}|\right)\,\dx+\frac{1}{4}\int_{B_R}\phi_{|\E \tilde u_\varepsilon|}(|\E u_{k,\varepsilon}-\E \tilde u_\varepsilon|)\,\dx\cr\cr &&
    +\frac{\tilde\varepsilon}{2}\int_{B_R}  |D^ku_{k,\varepsilon}|^2\,\dx+\frac{\tilde\varepsilon}{2}\int_{B_R}  |D^k u_{\varepsilon}|^2\,\dx\cr\cr 
    &&+c_\sigma\int_{B_R} \phi^*(|f_\varepsilon-f|)\,\dx+\sigma\int_{B_R}\phi(| u_{k,\varepsilon}- \tilde u_\varepsilon|)\,\dx\cr\cr
    &\le &
    c\int_{B_R} \phi_{|\E \tilde u_\varepsilon|}\left(|\E u-\E  u_{\varepsilon}|\right)\,\dx+\frac{1}{4}\int_{B_R}\phi_{|\E \tilde u_\varepsilon|}(|\E u_{k,\varepsilon}-\E \tilde u_\varepsilon|)\,\dx\cr\cr &&
    +\frac{\tilde\varepsilon}{2}\int_{B_R}  |D^ku_{k,\varepsilon}|^2\,\dx+\frac{\tilde\varepsilon}{2}\int_{B_R}  |D^k u_{\varepsilon}|^2\,\dx\cr\cr 
    &&+c_\sigma\int_{B_R} \phi^*(|f_\varepsilon-f|)\,\dx+c\sigma\int_{B_R}\phi(| Du_{k,\varepsilon}- D\tilde u_\varepsilon|)\,\dx \cr\cr&\le &
    c\int_{B_R} \phi_{|\E \tilde u_\varepsilon|}\left(|\E u-\E  u_{\varepsilon}|\right)\,\dx+\frac{1}{4}\int_{B_R}\phi_{|\E \tilde u_\varepsilon|}(|\E u_{k,\varepsilon}-\E \tilde u_\varepsilon|)\,\dx\cr\cr &&
    +\frac{\tilde\varepsilon}{2}\int_{B_R}  |D^ku_{k,\varepsilon}|^2\,\dx+\frac{\tilde\varepsilon}{2}\int_{B_R}  |D^k u_{\varepsilon}|^2\,\dx\cr\cr 
    &&+c_\sigma\int_{B_R} \phi^*(|f_\varepsilon-f|)\,\dx+c\sigma\int_{B_R}\phi(| \E u_{k,\varepsilon}- \E \tilde u_\varepsilon|)\,\dx.
\end{eqnarray*}  
Reabsorbing the second and the third integral in the right hand side by the left hand side of the previous estimate, we obtain
\begin{eqnarray*}
    &&{\frac{3}{4}}\int_{B_R}\phi_{|\E \tilde u_\varepsilon|}(|\E u_{k,\varepsilon}-\E \tilde u_\varepsilon|)\,\dx+\frac{\tilde\varepsilon}{2}\int_{B_R}  |D^ku_{k,\varepsilon}|^2\,\dx\cr\cr 
    &\le& 
    c\int_{B_R} \phi_{|\E \tilde u_\varepsilon|}\left(|\E u-\E  u_{\varepsilon}|\right)\,\dx+\frac{\tilde\varepsilon}{2}\int_{B_R}  |D^k u_{\varepsilon}|^2\,\dx\cr\cr 
    &&+c_\sigma\int_{B_R} \phi^*(|f_\varepsilon-f|)\,\dx+c\sigma\int_{B_R}\phi(| \E u_{k,\varepsilon}- \E \tilde u_\varepsilon|)\,\dx\cr\cr
    &\le& 
    c\int_{B_R} \phi_{|\E \tilde u_\varepsilon|}\left(|\E u-\E  u_{\varepsilon}|\right)\,\dx+\frac{\tilde\varepsilon}{2}\int_{B_R}  |D^k u_{\varepsilon}|^2\,\dx\cr\cr 
    &&+c_\sigma\int_{B_R} \phi^*(|f_\varepsilon-f|)\,\dx+c_\delta\cdot\sigma\int_{B_R}\phi_{|\E \tilde u_\varepsilon|}(| \E u_{k,\varepsilon}- \E \tilde u_\varepsilon|)\,\dx\cr\cr 
    &&+\delta\int_{B_R}\phi(|\E \tilde u_\varepsilon|)\,\dx,
\end{eqnarray*}
where, in the last inequality, we used \eqref{add-shift}. Choosing $\sigma=\frac{1}{2c_\delta}$, we can reabsorb the second last integral in the right hand side by the left hand side getting
\begin{eqnarray}\label{serve}
    &&{\frac{1}{4}}\int_{B_R}\phi_{|\E \tilde u_\varepsilon|}(|\E u_{k,\varepsilon}-\E \tilde u_\varepsilon|)\,\dx+\frac{\tilde\varepsilon}{2}\int_{B_R}  |D^ku_{k,\varepsilon}|^2\,\dx\cr\cr 
    &\le& 
    c\int_{B_R} \phi_{|\E \tilde u_\varepsilon|}\left(|\E u-\E  u_{\varepsilon}|\right)\,\dx+\frac{\tilde\varepsilon}{2}\int_{B_R}  |D^k u_{\varepsilon}|^2\,\dx\cr\cr 
    &&+c_\delta\int_{B_R} \phi^*(|f_\varepsilon-f|)\,\dx+\delta\int_{B_R}\phi(|\E \tilde u_\varepsilon|)\cr\cr
    &\le& 
    c_\delta\int_{B_R} \phi_{|\E  u|}\left(|\E u-\E  u_{\varepsilon}|\right)\,\dx+\frac{\tilde\varepsilon}{2}\int_{B_R}  |D^k u_{\varepsilon}|^2\,\dx\cr\cr 
    &&+c_\delta\int_{B_R} \phi^*(|f_\varepsilon-f|)\,\dx+\delta\int_{B_R}\phi(|\E \tilde u_\varepsilon|)\,\dx,
\end{eqnarray}
where we used the change of shift in  Lemma \ref{lem:change_of_shift} and \eqref{seconda}. Now, we observe that
\begin{eqnarray}\label{servebis}
    &&\int_{B_R} \phi_{|\E  u|}\left(|\E  u_{k,\varepsilon}-\E  u|\right)\, \dx\cr\cr&\le& c\int_{B_R} \phi_{|\E  u|}\left(|\E  u_{k,\varepsilon}-\E  u_{\varepsilon}|\right)\,\dx+c\int_{B_R} \phi_{|\E  u|}\left(|\E  u_{\varepsilon}-\E u|\right)\,\dx\cr\cr&\le &
 c\int_{B_R}\phi_{|\E \tilde u_\varepsilon|}(|\E u_{k,\varepsilon}-\E \tilde u_\varepsilon|)\,\dx+c\int_{B_R} \phi_{|\E  u|}\left(|\E  u_{\varepsilon}-\E u|\right)\,\dx,\end{eqnarray}
 where we used \eqref{sub-additivity-phi} and again the change of shift together with the last equivalence at \eqref{a-coercivity-2}. Combining \eqref{serve} and \eqref{servebis}, we get
 \begin{eqnarray*}
    &&\int_{B_R} \phi_{|\E  u|}\left(|\E  u_{k,\varepsilon}-\E  u|\right)\, \dx+\tilde\varepsilon\int_{B_R}  |D^ku_{k,\varepsilon}|^2\,\dx\cr\cr &\le &c_\delta\int_{B_R} \phi_{|\E  u|}\left(|\E u-\E  u_{\varepsilon}|\right)\,\dx+c\tilde\varepsilon\int_{B_R}  |D^k u_{\varepsilon}|^2\,\dx\cr\cr&&+c_\delta\int_{B_R} \phi^*(|f_\varepsilon-f|)\,\dx+c\cdot\delta\int_{B_R}\phi(|\E \tilde u_\varepsilon-\E u|)+c\cdot\delta\int_{B_R}\phi(|\E u|)\,\dx.
    \end{eqnarray*}
    By our choice of $\tilde \varepsilon$ at \eqref{sceltaepsilon} and since   \eqref{convmollificate} and \eqref{convmollificate1} hold, taking the limit first as $\varepsilon\to 0$ and then as $\delta\to 0$ in previous estimate,  we deduce that
    \begin{equation}\label{convforte2}\lim_{\varepsilon\to 0}\tilde\varepsilon\int_{B_R}  |D^ku_{k,\varepsilon}|^2\,\dx=0\end{equation}
    and
\begin{equation}\label{secondolimite}\lim_{\varepsilon\to 0}\int_{B_R} \phi_{|\E  u|}\left(|\E  u_{k,\varepsilon}-\E  u|\right)\, \dx=0. 
\end{equation}
As a consequence of Korn's inequality \eqref{korn-0}, from the last formula, we also have
$$\lim_{\varepsilon\to 0}\int_{B_R} \phi_{|\E  u|}\left(|D  u_{k,\varepsilon}-D  u|\right)\, \dx=0. $$
 Now observe that, by \eqref{add-shift}, for every $\delta>0$ there exists a constant $c_\delta>0$ such  that the following inequality holds
  $$ \int_{B_R} \phi\left(|\E  u_{k,\varepsilon}-\E  u|\right)\, \dx\le c_\delta\int_{B_R} \phi_{|\E  u|}\left(|\E  u_{k,\varepsilon}-\E  u|\right)\, \dx+ \delta\int_{B_R}|V(\E u)|^2\,\dx. $$
  Therefore, taking again the limit first as $\varepsilon\to 0$ and then as $\delta\to 0$ , from \eqref{secondolimite} we obtain
  \begin{equation*}\lim_{\varepsilon\to 0}\int_{B_R} \phi\left(|\E  u_{k,\varepsilon}-\E  u|\right)\, \dx=0. 
\end{equation*}
and so, recalling\eqref{limv}, we deduce that $u=v$ a.e. in $B_R$. Observing that
    \begin{equation}\label{convforte}\lim_{\varepsilon\to 0}\int_{B_R}\phi(|\E  u_{k,\varepsilon}|)\,\dx\le c\int_{B_R}\phi(|\E u|)\,\dx,\end{equation}
    and 
     \begin{equation}\label{convfortebis}\lim_{\varepsilon\to 0}\int_{B_R}\phi(|D  u_{k,\varepsilon}|)\,\dx\le c\int_{B_R}\phi(|D u|)\,\dx,\end{equation}
    taking the limit as $\varepsilon\to 0$ in estimate \eqref{estdersecenunciato}, by \eqref{convforte2}, \eqref{convforte} and \eqref{convfortebis}, we get
    \begin{equation*}
   \int_{B_{\rho}}|D( V(\E u))|^2\,\dx\le 
     c\left(1+\frac{k^2}{( r- \rho)^2}\right)\int_{B_{\tilde r}}\phi\left(|Du|\right)\,\dx+ck\int_{B_{ r}}\phi^*(|Df|)\,\dx,
 \end{equation*}
  which gives the conclusion.

 \bigskip
 
  \par\noindent {\bf Data availability statement.} Data sharing not applicable to this article as no datasets were generated or analysed during the current study.

\section*{Compliance with Ethical Standards}
\label{conflicts}

\smallskip
\par\noindent
{\bf Funding}. This research was partly funded by:
 \\ (i)  GNAMPA Project 2025, grant number E5324001950001, ``Regolarità di soluzioni di equazioni paraboliche a crescita nonstandard degeneri'' (F. Giannetti, A. Passarelli di Napoli);
 \\ (ii) GNAMPA Project 2026, grant number E53C25002010001 ``Esistenza e regolarita' per soluzioni di equazioni ellittiche e paraboliche anisotrope''( A. Passarelli di Napoli);
 \\ (iii) GNAMPA Project 2026, grant number E53C25002010001 ``EDP e Applicazioni: Dinamica dei fluidi e teoria spettrale''(F. Giannetti);
\\(iv)   Centro Nazionale per la Mobilità Sostenibile (CN00000023) - Spoke 10 Logistica Merci, grant number E63C22000930007,  funded by PNRR (A. Passarelli di Napoli).

\bigskip
\par\noindent
{\bf Conflict of Interest}. The authors declare that they have no conflict of interest.
\end{proof}

\end{document}